\documentclass[ECP]{ejpecp} 

\usepackage{pgfplots}

\SHORTTITLE{Convergence of complex martingales in the BRW} 

\TITLE{Convergence of complex martingales in the\\ branching random walk: the boundary.} 

\AUTHORS{%
	Konrad~Kolesko\footnote{Technical University of Darmstadt, Germany.
	\EMAIL{kolesko@mathematik.tu-darmstadt.de}}
  \and 
	Matthias~Meiners\footnote{Technical University of Darmstadt, Germany.
	\EMAIL{meiners@mathematik.tu-darmstadt.de}}}


\AMSSUBJ{60J80} 

\SUBMITTED{2016} 

\VOLUME{0}
\YEAR{2016}
\PAPERNUM{0}
\DOI{vVOL-PID}

\ABSTRACT{
Biggins [Uniform convergence of martingales in the branching random walk.
{\em Ann. Probab.}, 20(1):137--151, 1992]
proved local uniform convergence of additive martingales in $d$-dimensional supercritical branching random walks
at complex parameters $\lambda$ from an open set $\Lambda \subseteq \C^d$.
We investigate the martingales corresponding to parameters from the boundary $\partial \Lambda$ of $\Lambda$.
The boundary can be decomposed into several parts.
There may be a part of the boundary, on which the martingales do not exist,
on other parts it exists, but diverges or vanishes in the limit.
In the remaining part, there is convergence to a non-degenerate limit.
The arguments that give this convergence also apply in $\Lambda$
and require weaker moment assumptions than the ones used by Biggins.
}

\newcommand{\defeq}{\vcentcolon=}

\newcommand{\N}{\mathbb{N}}
\newcommand{\Z}{\mathbb{Z}}

\newcommand{\R}{\mathbb{R}}

\newcommand{\C}{\mathbb{C}}
\newcommand{\Gen}{\mathcal{G}}
\newcommand{\Dom}{\mathcal{D}}
\newcommand{\I}{\mathcal{I}}
\newcommand{\V}{\mathbb{V}}

\newcommand{\imag}{\mathrm{i}}
\newcommand{\Real}{\mathrm{Re}}
\newcommand{\Imag}{\mathrm{Im}}

\newcommand{\comp}{\mathsf{c}}

\newcommand{\Sim}{\mathbb{S}\mathit{(d)}}

\newcommand{\Prob}{\mathbb{P}}

\DeclareMathOperator{\inner}{\mathrm{int}}

\newcommand{\E}{\mathbb{E}}
\newcommand{\A}{\mathcal{A}}

\newcommand{\F}{\mathcal{F}}
\newcommand{\calZ}{\mathcal{Z}}
\newcommand{\1}{\mathbb{1}}

\newcommand{\du}{\mathrm{d} \mathit{u}}
\newcommand{\dx}{\mathrm{d} \mathit{x}}

\begin{document}



\section{Introduction}

Biggins \cite{Biggins:1992} proved local uniform convergence
of additive martingales in a supercritical branching random walk on $\R^d$
at complex parameters within a certain open set $\Lambda \subseteq \C^d$.
He used the results obtained to derive a local large deviation result
for the point process of the positions in the $n$th generation as $n \to \infty$.

In some situations, the arguments from \cite{Biggins:1992} cover parts of the boundary $\partial \Lambda$ of $\Lambda$,
but typically only a proper, possibly empty, subset of $\partial \Lambda$.
However, the ideas and results required to deal with the boundary are available in the literature,
but spread over different papers \cite{Aidekon+Shi:2014,Buraczewski+Kolesko:2014,Iksanov+Meiners:2015}
and not directly applicable.
In this paper, we gather these techniques and results
and provide a complete treatment (up to mild moment assumptions)
of the convergence of additive martingales on the boundary $\partial \Lambda$.

Besides its value in the study of large deviation results for the branching random walk
and its intrinsic interest,
there is further motivation to study the convergence of additive martingales at complex parameters,
particularly on the boundary $\partial \Lambda$.

First, in the recent applied probability literature,
there are several examples of limit theorems, in which the limiting behavior
of a quantity of interest is described by the solution to a complex smoothing equation.
This solution can always be chosen as the limit of a suitable additive martingale at a complex parameter,
see \cite{Meiners+Mentemeier:2017} for a discussion and a collection of examples including fragmentation
processes and P\'olya urns.
Understanding the convergence of additive martingales at the boundary $\partial\Lambda$
is essential for the study of critical smoothing equations.
This is our major motivation for writing the note at hand.

Second,
the additive martingales are intimately connected with cascade measures,
processes that have initially been introduced by Mandelbrot as statistical models for turbulence
\cite{Mandelbrot:1972,Mandelbrot:1974}.
The parameters on the boundary $\partial \Lambda$ correspond to boundaries between different phases
of the cascade model, see  e.g.\ \cite{Barral+al:2010,Madaule+al:2015} and the references therein.

\section{Main results}	\label{sec:main results}

\paragraph{Model description.}

We consider a branching random walk in $\R^d$ where $d \in \N = \{1,2,\ldots\}$.
The process starts with an initial ancestor at the origin.
The ancestor forms generation $0$ of the process
and produces offspring placed on $\R^d$ at the points of a point process $\calZ = \sum_{j=1}^N \delta_{X_j}$
with intensity measure $\mu$.
The children of the ancestor form the first generation of the process.
Each member of the first generation has children with
positions relative to their parent's position given by an independent copy of $\calZ$,
and so on.
We suppose that the branching random walk is supercritical, that is,
$\mu(\R^d) = \E[N] > 1$.

More formally, let $\I := \bigcup_{n \geq 0} \N^n$ be the set of finite tuples of positive integers.
If $u=(u_1,\ldots,u_n) \in \N^n$ and $v=(v_1,\ldots,v_m) \in \N^m$,
we write $u_1 \ldots u_n$ for $u$ and $uv$ for $(u_1,\ldots,u_n,v_1,\ldots,v_m)$.
Further, we write $u|_k$ for $(u_1,\ldots,u_{k \wedge n})$, $k \in \N_0$.

The ancestor is identified with the empty tuple $\varnothing$ and its position is $S(\varnothing) = 0$.
On some probability space $(\Omega,\A,\Prob)$, let $(\calZ(u))_{u \in \I}$ be a family of i.i.d.\ copies of $\calZ$.
For ease of notation, we assume $\calZ(\varnothing)=\calZ$.
We write $\calZ(u) = \sum_{i=1}^{N(u)} X_i(u)$, where $N(u) = \calZ(u)(\R^d)$, $u \in \I$.
Then $\Gen_0 \defeq \{\varnothing\}$ is generation $0$ of the process
and, recursively,
\begin{equation*}
\Gen_{n+1} \defeq \{uj \in \N^{n+1}: u \in \Gen_n \text{ and } 1 \leq j \leq N(u)\}
\end{equation*}
is generation $n+1$ of the process, $n \in \N_0 \defeq \N \cup \{0\}$.
Define the set of all individuals by $\Gen \defeq \bigcup_{n \in \N_0} \Gen_n$.
We write $|u|=n$ for $u \in \Gen_n$ and $|u|<n$ if $u \in \Gen_k$ for some $k<n$. 
The position of an individual $u = u_1 \ldots u_n \in \Gen_n$ is given by
\begin{equation*}
S(u)	\defeq	X_{u_1}(\varnothing) + \ldots + X_{u_n}(u_1 \ldots u_{n-1}).
\end{equation*}
The point process of the $n$th generation people will be denoted by $\calZ_n$, that is,
\begin{equation*}
\calZ_n = \sum_{|u|=n} \delta_{S(u)}.
\end{equation*}
The sequence of point processes $(\calZ_n)_{n \in \N_0}$ is then called {\it branching random walk}.

The multivariate Laplace transform $m$ of $\mu$ is denoted by
\begin{equation*}	\textstyle
m(\lambda) = \int e^{-\lambda x} \, \mu(\dx),
\end{equation*}
where $\lambda \in \C^d$ and $\lambda = \theta + \imag \eta$ with $\theta, \eta \in \R^d$.
(We adopt the convention from \cite{Biggins:1992} and always write $\theta$ for $\Real(\lambda)$ and $\eta$ for $\Imag(\lambda)$.)
We are only interested in those $\lambda$ for which $m(\lambda)$ is well-defined, 
i.\,e., $\lambda$ from the set
\begin{equation*}	\textstyle
\Dom = \{\lambda \in \C^d: m(\lambda) \text{ converges}\} = \{\theta \in \R^d: m(\theta) < \infty\} + \imag \R^d.
\end{equation*}
Throughout, we assume $\inner \Dom \not = \varnothing$.
Let $\F_0$ be the trivial $\sigma$-field and, for $n \in \N$,
\begin{equation*}	\textstyle
\F_n \defeq \sigma(\calZ(u): u \in \N^k \text{ for some } k < n).
\end{equation*}
Then, for $\lambda \in \Dom$ with $m(\lambda) \not = 0$, the family
\begin{equation*}	\textstyle
Z_n(\lambda) = m(\lambda)^{-n} \sum_{|u|=n} e^{-\lambda S(u)},	\quad	n \in \N_0
\end{equation*}
forms a complex martingale with respect to $(\F_n)_{n \in \N_0}$.

\paragraph{Point of departure.}

Biggins \cite[Theorem 1]{Biggins:1992} proved that if
\begin{equation}	\label{eq:gamma moment Z_1(theta)}	\textstyle
\E[Z_1(\theta)^\gamma] < \infty	\quad	\text{for some } \gamma \in (1,2]
\end{equation}
and
\begin{equation}	\label{eq:contraction condition}	\textstyle
\frac{m(p\theta)}{|m(\lambda)|^p} < 1	\quad	\text{for some } p \in (1,\gamma],
\end{equation}
then $(Z_n(\lambda))_{n \geq 0}$ converges almost surely and in $p$th mean to
a limit variable $Z(\lambda)$.
What is more, Biggins \cite[Theorem 2]{Biggins:1992} proved that
this convergence is locally uniform (almost surely and in mean)
on the set $\Lambda = \bigcup_{\gamma \in (1,2]} \Lambda_{\gamma}$
where $\Lambda_\gamma = \Lambda_{\gamma}^1 \cap \Lambda_{\gamma}^3$
and, for $\gamma \in (1,2]$,
\begin{equation*}	\textstyle
\Lambda_{\gamma}^1 =
\inner\{\lambda \in \Dom: \E[Z_1(\theta)^\gamma] < \infty\}
\quad	\text{and}	\quad
\textstyle
\Lambda_{\gamma}^3 =
\inner\big\{\lambda \in \Dom: \inf_{1 \leq p \leq \gamma} \frac{m(p \theta)}{|m(\lambda)|^p} < 1\big\}.
\end{equation*}

\paragraph{The boundary of $\Lambda$.}

We decompose $\partial \Lambda$ into several parts.
The first part is $\partial \Lambda^0 \defeq \partial \Lambda \cap \Dom^\comp$.
Notice that $\partial \Lambda^0$ may be non-empty, see the example in Section \ref{sec:discussion and example}.
The martingale $(Z_n(\lambda))_{n \geq 0}$ is not defined on $\partial \Lambda^0$,
so we will exclude this set from the further discussion.
We introduce a weaker form of \eqref{eq:contraction condition}, namely,
\begin{equation*}	\tag{C1}	\label{eq:characteristic index}	\textstyle
\frac{m(\alpha \theta)}{|m(\lambda)|^\alpha} = 1
\ \text{ and }\ 
\E\big[\sum_{|u|=1} \theta S(u) \frac{e^{-\alpha \theta S(u)}}{|m(\lambda)|^\alpha}\big] \geq -\log (|m(\lambda)|)
\quad	\text{for some } \alpha \in [1,2],
\end{equation*}
and, additionally, the following moment condition:
\begin{equation*}	\tag{C2}	\label{eq:logarithmic moment condition}
\E[|Z_1(\lambda)|^\alpha \log_+^{2+\epsilon} (|Z_1(\lambda)|)] < \infty	\quad	\text{ for some } \epsilon > 0.
\end{equation*}
Subject to the moment condition \eqref{eq:logarithmic moment condition}, there is convergence almost surely and in mean
of the martingales at $\lambda$ from
\begin{equation*}
\partial \Lambda^{(1,2)} \defeq \{\lambda \in \partial \Lambda \cap \Dom: \eqref{eq:characteristic index}
\text{ holds with } \alpha \in (1,2)\},
\end{equation*}
see Theorem \ref{Thm:alpha in (1,2)} below.
On the set
\begin{equation*}
\partial \Lambda^1
\defeq \{\lambda \in \Dom \cap \partial \Lambda: \eqref{eq:characteristic index} \text{ holds with } \alpha = 1\},
\end{equation*}
we have from the first condition in \eqref{eq:characteristic index} with $\alpha=1$
that $m(\theta) = |m(\lambda)|$.
Hence, $Z_1(\lambda) = Z_1(\theta)$ almost surely.
Consequently, $(Z_n(\lambda))_{n \geq 0}$ is a real martingale for $\lambda \in \partial \Lambda^1$.
Whether or not the additive martingale in the branching random walk converges
in the real case is known from Biggins' martingale convergence theorem
\cite{Alsmeyer+Iksanov:2009,Biggins:1977,Lyons:1997}.
We therefore omit the treatment of $\partial \Lambda^1$ in what follows.
Further, typically (see Proposition \ref{Prop:alpha=2} for the details), there is no convergence on
\begin{align*}
\partial \Lambda^2
&\defeq \{\lambda \in \Dom \cap \partial \Lambda: \eqref{eq:characteristic index}
\text{ holds with } \alpha = 2\}	\\
\text{and}	\qquad
\partial \Lambda^3
&\defeq \{\lambda \in \Dom \cap \partial \Lambda: \E[Z_1(\lambda)^\gamma] = \infty \text{ for every } \gamma > 1\}.
\end{align*}
In most situations, it will hold that
\begin{equation}	\label{eq:boundary exhausted}	\textstyle
\partial \Lambda = \partial \Lambda^0 \cup \partial \Lambda^1 \cup \partial \Lambda^{(1,2)} \cup \partial \Lambda^2 \cup \partial \Lambda^3,
\end{equation}
i.e., the sets defined above exhaust $\partial \Lambda$.
There is a discussion including a set of (mild) conditions that ensure \eqref{eq:boundary exhausted} to hold
in Section \ref{sec:discussion and example} below.

\paragraph{Main theorems.}

To unburden the notation, we fix $\lambda \in \Dom$
and set $L(u) \defeq m(\lambda)^{-n} e^{-\lambda S(u)}$ if $u \in \Gen_n$ for some $n \in \N_0$,
and $L(u) \defeq 0$, otherwise.
We write $Z_n$ for $Z_n(\lambda)$, $n \in \N$
and $Z$ for $Z(\lambda)$ if the latter exists.
By construction, $(Z_n)_{n \geq 0}$ is a complex martingale with
\begin{equation}	\label{eq:unit mean} \textstyle
\E[Z_1] = 1.
\end{equation}
To avoid trivialities, we assume that $\Prob(Z_1 = 1) < 1$.
Condition \eqref{eq:characteristic index} in the simplified notation becomes
\begin{equation*}	\tag{\ref{eq:characteristic index}}	\textstyle
\E\big[\sum_{|u|=1} |L(u)|^{\alpha}\big] = 1
\text{ and }
\E\big[\sum_{|u|=1} |L(u)|^{\alpha} \log (|L(u)|) \big] \leq 0
\quad	\text{ for some } \alpha \in [1,2].
\end{equation*}
Condition \eqref{eq:logarithmic moment condition} in the simplified notation reads
\begin{equation*}	\tag{\ref{eq:logarithmic moment condition}}
\E[|Z_1|^\alpha \log_+^{2+\epsilon} (|Z_1|)] < \infty	\quad	\text{ for some } \epsilon > 0.
\end{equation*}
Sometimes, we will refer to the following condition:
\begin{equation*}	\tag{C3}	\label{eq:m(vartheta)<infty}	\textstyle
\E\big[\sum_{|u|=1} |L(u)|^{\vartheta}\big] < \infty	\quad	\text{for some } \vartheta \in [0,\alpha).
\end{equation*}
We further define $W_n \defeq \sum_{|u|=n} |L(u)|^{\alpha}$, $n \in \N_0$.
Then, by \eqref{eq:characteristic index}, $(W_n)_{n \geq 0}$ is a nonnegative martingale.
The martingale convergence theorem and Fatou's lemma give $W_n \to W$ almost surely
for a nonnegative random variable $W$ with $\E[W] \in \{0,1\}$.
Whether $\E[W]=0$ or $\E[W]=1$ is known from Biggins' martingale convergence theorem
\cite{Alsmeyer+Iksanov:2009,Biggins:1977,Lyons:1997}.

The following theorem is the main result of the paper.
It gives convergence of the additive martingales to non-degenerate limits on $\partial \Lambda^{(1,2)}$.

\begin{theorem}	\label{Thm:alpha in (1,2)}
Suppose that \eqref{eq:characteristic index} and \eqref{eq:logarithmic moment condition} hold with $\alpha \in (1,2)$.
Then $(Z_n)_{n \geq 0}$ converges almost surely and in $L^p$ for every $p < \alpha$
to a non-degenerate limit $Z$.
\end{theorem}

The following propositions are essentially contained in \cite{Iksanov+Meiners:2015}
and provide sufficient conditions for the divergence of the additive martingales on $\partial \Lambda^2$ and $\partial \Lambda^3$, respectively.

\begin{proposition}	\label{Prop:alpha=2}
Suppose that $\Prob(N < \infty)=1$
and that \eqref{eq:characteristic index} holds with $\alpha = 2$.
Then each of the following two conditions is sufficient for $(Z_n)_{n \geq 0}$
not to converge in probability.
\begin{itemize}
	\item[(i)]
		$\E[\sum_{|u|=1} |L(u)|^{2} \log (|L(u)|)] \in (-\infty,0)$
		and $\E[W_1 \log_+ W_1] < \infty$,
	\item[(ii)]
		$\E[\sum_{|u|=1} |L(u)|^{2} \log (|L(u)|)] = 0$, \eqref{eq:m(vartheta)<infty} holds
		and
		\begin{align*}	\textstyle
		\E[\sum_{|u|=1} |L(u)|^2 \log^2 (|L(u)|)]<\infty,\  
		\E[W_1 \log_+^2 (W_1)]<\infty
		\text{ and }
		\E[\widetilde{W_1} \log_+ (\widetilde{W_1})]<\infty
		\end{align*}
		where $\widetilde{W}_1 \defeq \sum_{|u|=1} |L(u)|^2 \log_- (|L(u)|)$.
\end{itemize}
\end{proposition}

\begin{proposition}	\label{Prop:E[Z_1^gamma]=infty f.a. gamma>1}
Suppose that $\Prob(N < \infty)=1$ and that \eqref{eq:characteristic index} with $\alpha > 1$
and \eqref{eq:m(vartheta)<infty} hold.
If $\E[|Z_1|^p] = \infty$ for some $p \in [1,\alpha)$, then $(Z_n)_{n \geq 0}$ does not converge in probability.
\end{proposition}

\begin{remark}	\label{Rem:N<infty}
In both propositions, we require $\Prob(N<\infty)=1$.
This is because their proofs are based on arguments from \cite{Alsmeyer+Meiners:2013,Iksanov+Meiners:2015}
involving complex multiplicative martingales and convergence of triangular arrays.
It may be possible, but certainly tedious, to extend those arguments to the case $\Prob(N=\infty)>0$.
As we want to keep the presentation short and accessible,
we refrain from trying to remove the assumption.
\end{remark}

The rest of the paper is organized as follows.
In Section \ref{sec:discussion and example},
we give a brief discussion of the shape of $\Lambda$, its boundary
and the parts in which the boundary can be divided.
We further give an example to illustrate our results.
Section \ref{sec:proofs} contains the proofs of our results,
while Section \ref{sec:similarities} contains extensions of the main results
to a more general, multidimensional situation.
Finally, there is an appendix comprising an auxiliary result required in the proof of Theorem \ref{Thm:alpha in (1,2)}.

\section{Discussion and examples}	\label{sec:discussion and example}

It is illustrative to first consider examples.

\paragraph{Examples.}

We begin with an example which strongly 	reminiscent of the situation studied in \cite{Madaule+al:2015}.
We also refer to \cite{Lacoin+al:2015}, where the problem of convergence on $\Lambda_{\gamma}^{(1,2)}$ is studied
in the different context of Gaussian multiplicative chaos.

\begin{example}[The Gaussian case with binary splitting]
Consider a branching random walk with independent standard Gaussian increments
and binary splitting, i.\,e., $\calZ = \delta_{X_1}+\delta_{X_2}$
where $X_1,X_2$ are i.i.d.\ random variables with standard normal laws.
Then $m(\lambda) = 2 \exp(\lambda^2/2)$ for all $\lambda \in \C$.
For every $\theta \in \R$ and every $\gamma > 1$, we have
\begin{equation*}	\textstyle
\E[Z_1(\theta)^\gamma] = \frac{1}{m(\theta)^{\gamma}} \E[(e^{-\theta X_1} + e^{-\theta X_2})^\gamma]
\leq \frac{2^{\gamma}}{m(\theta)^{\gamma}} \E[e^{-\theta \gamma X_1}] = \frac{2^{\gamma}m(\theta \gamma)}{m(\theta)^\gamma} < \infty.
\end{equation*}
Hence
$\Lambda = \{\lambda \in \C: m(p\theta)/|m(\lambda)|^p < 1 \text{ for some } p \in (1,2]\}$.
Thus, $\lambda \in \Lambda$ if and only if there exists some $p \in (1,2]$ with $m(p\theta)/|m(\lambda)|^p < 1$.
The latter inequality is equivalent to
\begin{equation}	\label{eq:binary Gaussian condition}	\textstyle
(1-p) 2 \log 2 + p^2 \theta^2 - p(\theta^2-\eta^2) < 0.
\end{equation}
By symmetry, it suffices to consider $\theta, \eta \geq 0$ only.
Next notice that $\sup\{\theta: \lambda \in \Lambda\} = \sqrt{2 \log 2}$.
For fixed $\theta \in [0,\sqrt{2 \log 2}]$, making \eqref{eq:binary Gaussian condition} explicit in $\eta^2$ gives:
\begin{equation*}	\label{eq:binary Gaussian condition explicit in eta}	\textstyle
\eta^2 < \frac{p-1}{p} 2 \log 2 - (p-1) \theta^2.
\end{equation*}
The right-hand side assumes its maximum (as a function of $p \in (1,2]$)
at $p = (\sqrt{2 \log 2}/\theta) \wedge 2$
giving $\eta < \sqrt{2 \log 2} - \theta$ for all $\sqrt{\log 2}/\sqrt{2} \leq \theta < \sqrt{2 \log 2}$.
For $0 \leq \theta \leq \sqrt{(\log 2)/2}$, we get $\theta^2+\eta^2 < \log 2$.
In conclusion, we get the shape depicted in Figure \ref{fig:Gaussian Lambda} for $\Lambda$.
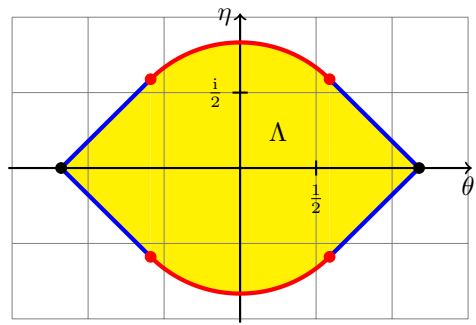
\begin{figure}[h]
\begin{center}
\begin{tikzpicture}[scale=1]
		\fill[yellow] plot[domain=-2*sqrt(2*ln(2)):-2*sqrt(ln(2)/2)] (\x, {2*sqrt(2*ln(2))+\x)})
					-- plot[domain=-2*sqrt(ln(2)/2):-2*sqrt(2*ln(2))] (\x, {-2*sqrt(2*ln(2))-\x)});
		\fill[yellow] plot[domain=2*sqrt(ln(2))/sqrt(2):2*sqrt(2*ln(2))] (\x, {2*sqrt(2*ln(2))-\x)})
					-- plot[domain=2*sqrt(2*ln(2)):2*sqrt(ln(2))/sqrt(2)] (\x, {-2*sqrt(2*ln(2))+\x)});
		\fill[yellow] plot[domain=-2*sqrt(ln(2)/2):2*sqrt(ln(2)/2)] (\x, {2*sqrt(ln(2)-(\x/2)*(\x/2))})
					-- plot[domain=2*sqrt(ln(2)/2):-2*sqrt(ln(2)/2)] (\x, {-2*sqrt(ln(2)-(\x/2)*(\x/2))});
		\draw [help lines] (-3,-2) grid (3,2);
		\draw [thick,->] (0,-2.05) -- (0,2.05);
		\draw [thick,->] (-3.05,0) -- (3.05,0);
		\draw (3,0) node[below]{$\theta$};
		\draw (0,2) node[left]{$\eta$};
		\draw (0.5,0.25) node[above]{$\Lambda$};
		\draw [ultra thick, domain=-2*sqrt(ln(2)/2):2*sqrt(ln(2)/2), samples=128, color=red] plot(\x, {2*sqrt(ln(2)-(\x/2)*(\x/2))});
		\draw [ultra thick, domain=-2*sqrt(ln(2)/2):2*sqrt(ln(2)/2), samples=128, color=red] plot(\x, {-2*sqrt(ln(2)-(\x/2)*(\x/2))});
		\draw [ultra thick, domain=2*sqrt(ln(2))/sqrt(2):2*sqrt(2*ln(2)), samples=128, color=blue] plot(\x, {2*sqrt(2*ln(2))-\x)});
		\draw [ultra thick, domain=2*sqrt(ln(2))/sqrt(2):2*sqrt(2*ln(2)), samples=128, color=blue] plot(\x, {-2*sqrt(2*ln(2))+\x)});
		\draw [ultra thick, domain=-2*sqrt(2*ln(2)):-2*sqrt(ln(2))/sqrt(2), samples=128,color=blue] plot(\x, {2*sqrt(2*ln(2))+\x)});
		\draw [ultra thick, domain=-2*sqrt(2*ln(2)):-2*sqrt(ln(2))/sqrt(2), samples=128, color=blue] plot(\x, {-2*sqrt(2*ln(2))-\x)});
		\filldraw[color=black] (2.35482,0) circle (2pt);
		\filldraw[color=black] (-2.35482,0) circle (2pt);
		\filldraw[color=red] (1.17741,1.17741) circle (2pt);
		\filldraw[color=red] (1.17741,-1.17741) circle (2pt);
		\filldraw[color=red] (-1.17741,1.17741) circle (2pt);
		\filldraw[color=red] (-1.17741,-1.17741) circle (2pt);
		\draw[thick] (1,0.1) --(1,-0.1) node[below]{\small $\frac12$};
		\draw[thick] (0.1,1) -- (-0.1,1) node[left]{\small $\frac\imag2$};
\end{tikzpicture}
\caption{The figure shows $\Lambda$ (in yellow) and $\partial \Lambda$ (in red, blue and with two black dots)
for the branching random walk with binary splitting and independent standard Gaussian increments.
Convergence of the additive martingales for $\lambda$ from the yellow phase follows from \cite[Theorem 1]{Biggins:1992},
but also from our Theorem \ref{Thm:alpha in (1,2)}.
The black dots form $\partial \Lambda^1$ and correspond to the real martingale in what is called the boundary case in the literature.
There is no convergence at the black dots.
The blue lines form $\partial \Lambda^{(1,2)}$ and thus correspond to the case $1 < \alpha < 2$.
Theorem \ref{Thm:alpha in (1,2)} yields that there is convergence to a nontrivial limit on the blue lines.
The red lines including the endpoints form $\partial \Lambda^2$,
which is dealt with in Proposition \ref{Prop:alpha=2}.
Parts (a) and (b) of the proposition yield that there is no convergence on the red arcs without the endpoints
and in the endpoints, respectively.}
\label{fig:Gaussian Lambda}
\end{center}
\end{figure}
\end{example}

We continue with a somewhat pathological example in which $\partial \Lambda \cap \Dom^\comp \not = \varnothing$.

\begin{example}	\label{Exa:pathological}
Let $\calZ = \sum_{k=1}^N \delta_{X_k}$ with $\Prob(N = n(n+1)) = \frac{1}{n(n+1)}$ for all $n \in \N$
and $\Prob(X_k = n \mid N=n(n+1)) = 1$ for $k=1,\ldots,n(n+1)$. Then, for $\theta > 0$,
$m(\theta) = e^{-\theta}/(1-e^{-\theta})$.
It is easily checked that $\E[Z_1(\theta)^2] < \infty$ for all $\theta > 0$.
We now explicitly determine $\Lambda$. To this end,
notice that any $\lambda$ with $\theta > 0$ is in $\Lambda$ iff for some $p\in(1,2]$, we have $m(p\theta)/|m(\lambda)|^p < 1$,
equivalently,
\begin{equation*}	\textstyle
|1-e^{-(\theta+\imag \eta)}|^p = (|1-e^{-(\theta+\imag \eta)}|^2)^{p/2} = (1-2e^{-\theta} \cos\eta + e^{-2\theta})^{p/2} < 1-e^{-p\theta}.
\end{equation*}
Making this inequality explicit in $\cos\eta$ results in
\begin{equation*}	\textstyle
\frac{1}{2} \big(e^{\theta} - (1-e^{-p\theta})^{p/2} e^{\theta} + e^{-\theta}\big) < \cos\eta.
\end{equation*}
Since $p=2$ is the minimizer for the left-hand side as a function of $p \in [1,2]$,
we have $\lambda \in \Lambda$ iff $e^{-\theta} < \cos\eta$.
Thus, since $m(0) = \infty$, it holds that
\begin{equation*}	\textstyle
\Lambda = \{\theta + \imag \eta: \theta > 0, e^{-\theta} < \cos\eta\}
= \bigcup_{n \in \Z} \big(2\pi \imag n + \{\theta + \imag \eta: \theta > 0, |\eta| < \frac\pi2, e^{-\theta} < \cos\eta\}\big).
\end{equation*}

\begin{figure}[h]
\begin{center}
\begin{tikzpicture}[scale=0.4]
		\fill[yellow,samples=256] plot[domain=0:15] (\x, {acos(exp(-\x))/180*pi)})
					-- plot[domain=15:0] (\x, {-acos(exp(-\x))/180*pi)});
		\draw [ultra thick, domain=0:15, samples=256, color=red] plot(\x, {acos(exp(-\x))/180*pi});
		\draw [ultra thick, domain=0:15, samples=256, color=red] plot(\x, {-acos(exp(-\x))/180*pi});
		\fill[yellow,samples=256] plot[domain=0:15] (\x, {acos(exp(-\x))/180*pi)+2*pi})
					-- plot[domain=15:0] (\x, {-acos(exp(-\x))/180*pi)+2*pi});
		\draw [ultra thick, domain=0:15, samples=256, color=red] plot(\x, {acos(exp(-\x))/180*pi+2*pi});
		\draw [ultra thick, domain=0:15, samples=256, color=red] plot(\x, {-acos(exp(-\x))/180*pi+2*pi});
		\draw [help lines] (-0.2,-2) grid (15,8);
		\draw [thick,->] (0,-2.05) -- (0,8.05);
		\draw [thick,->] (-0.25,0) -- (15.05,0);
		\draw [-] (-0.1,pi/2) -- (15.05,pi/2);
		\draw [-] (-0.1,3*pi/2) -- (15.05,3*pi/2);
		\draw (0,pi/2) node[left]{$\frac\pi2$};
		\draw (0,3*pi/2) node[left]{$\frac{3\pi}2$};
		\draw [-] (-0.1,-pi/2) -- (15.05,-pi/2);
		\draw (0,-pi/2) node[left]{$-\frac\pi2$};		
		\draw (15,0) node[below]{$\theta$};
		\draw (0,8) node[left]{$\imag \eta$};
		\filldraw[color=green] (0,0) circle (5pt);
		\filldraw[color=green] (0,2*pi) circle (5pt);
\end{tikzpicture}
\caption{The figure shows $\Lambda$ (in yellow) and $\partial \Lambda$ (the red curves and green dots).
As $\calZ$ is concentrated on $\Z$, the Laplace transform $m$ is $2\pi\imag$-periodic,
hence $\Lambda$ consists of a countable family of shifted copies of the connected part of $\Lambda$
intersecting the halfline $\{\lambda: \theta > 0, \eta = 0\}$.
Convergence of the additive martingales for $\lambda$ from the yellow phase follows from \cite[Theorem 1]{Biggins:1992}
and Theorem \ref{Thm:alpha in (1,2)}.
The green dots correspond to the domain $\partial \Lambda^0 = \partial \Lambda \cap \Dom^\comp$.
The martingale is not defined on this set.
The red curves form $\partial \Lambda^2$, i.e., they correspond to the case $\alpha=2$.
There is no convergence on the red curves by Proposition \ref{Prop:alpha=2}(a) (there is some checking required to see that the proposition applies).}
\label{fig:lattice Lambda}
\end{center}
\end{figure}
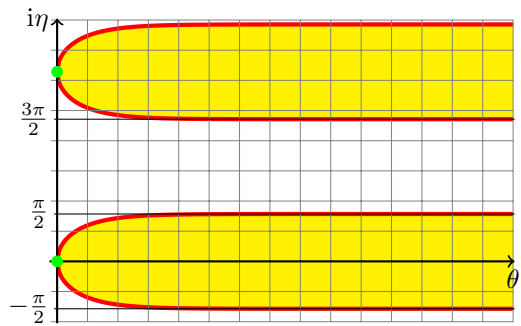

\end{example}

\paragraph{Discussion of the assumptions.}

There is a discussion of the shape of $\Lambda$ on p.\;141 of \cite{Biggins:1992}.
Here, we want to confine ourselves to explaining why one can expect
that \eqref{eq:boundary exhausted} holds, i.e.,
that the boundary is typically exhausted by
$\partial \Lambda^0 \cup \partial \Lambda^1 \cup \partial \Lambda^{(1,2)} \cup \partial \Lambda^2 \cup \partial \Lambda^3$.

\begin{lemma}	\label{Lem:boundary of Lambda_gamma}
Let $\lambda \in \partial \Lambda_\gamma$ for some $\gamma \!\in\! (1,2]$.
If\ ${\Prob(Z_1(\lambda) \in [0,\infty))<1}$, then \eqref{eq:characteristic index} holds with $\alpha \in (1,\gamma]$.
\end{lemma}
\begin{proof}
We conclude $m(\theta) < \infty$ from $\E[Z_1(\theta)^\gamma] < \infty$.
We further have
\begin{equation*}	\textstyle
m(\gamma \theta) = \E\big[ \sum_{|u|=1} e^{-\gamma \theta S(u)} \big] \leq \E\big[ \big(\sum_{|u|=1} e^{-\theta S(u)} \big)^\gamma\big]
= \E[Z_1(\theta)^\gamma] < \infty.
\end{equation*}
Define the functions, $p \mapsto f(p) \defeq m(p\theta)/|m(\lambda)|^p$
and $p \mapsto f_n(p) \defeq m(p\theta_n)/|m(\lambda_n)|^p$,
where $\lambda_n \in \Lambda_\gamma$ are such that $\lambda_n \to \lambda$.
Then $f,f_1,f_2,\ldots$ are finite and continuous on $[1,\gamma]$.
Further, $\lim_{n \to \infty} \lambda_n = \lambda$ implies $f_n \to f$ pointwise on $[1,\gamma]$
and hence $\inf_{1 \leq p \leq \gamma} f(p) \leq 1$.
Let $\alpha \in [1,\gamma]$ be minimal with $f(\alpha)=m(\alpha \theta)/|m(\lambda)|^\alpha = 1$.
This is the first condition of \eqref{eq:characteristic index}.
Clearly, $\alpha > 1$ since $\Prob(Z_1(\lambda) \in [0,\infty)) < 1$.
Thus, $f$ is differentiable at $\alpha$ (from the left if $\alpha = \gamma$)
with $f'(\alpha) \leq 0$,
which translates into the second condition of \eqref{eq:characteristic index}.
\end{proof}

The lemma explains the choice of $\partial \Lambda^{(1,2)}$.
In the situation of the lemma, \eqref{eq:logarithmic moment condition} is automatically fulfilled if $\gamma > \alpha$.
If $\alpha = \gamma$, we have $\E[|Z_1(\lambda)|^\alpha] \leq \E[Z_1(\theta)^\alpha] < \infty$
and \eqref{eq:logarithmic moment condition} thus constitutes only a very mild additional moment assumption.

\section{Proofs of the main results}	\label{sec:proofs}

\paragraph{Many-to-one lemma and auxiliary results for random walks.}

There is a well-known simple formula with far-reaching implications that connects
the branching random walk $(\calZ_n)_{n \in \N_0}$
with an associated standard random walk $(S_n)_{n \in \N_0}$ on $\R$.
This formula is sometimes called the many-to-one lemma and takes the following form here:
\begin{align}	\label{eq:many-to-one}
\E[f(S_0,\ldots,S_n)] = \E\bigg[\sum_{|u|=n} |L(u)|^{\alpha} f(0,-\log(|L(u|_1)|),\ldots,-\log(|L(u)|)) \bigg]
\end{align}
for all nonnegative Borel-measurable functions $f: \R^{n+1} \to \R$.
The formula is used in many (possibly all) papers on branching random walks.
We thus refrain from proving it here.
We just mention an important consequence of \eqref{eq:many-to-one},
namely, choosing $n=1$ and $f(x,y) = y$, whenever $S_1$ or $\sum_{|u|=1} |L(u)|^{\alpha} (-\log(|L(u)|))$ is quasi-integrable, we get
\begin{align}	\label{eq:E[S_1]}	\textstyle
\E[S_1] = \E\big[\sum_{|u|=1} |L(u)|^{\alpha} (-\log(|L(u)|))\big].
\end{align}

\paragraph{Proofs of Theorem \ref{Thm:alpha in (1,2)}, Proposition \ref{Prop:alpha=2} and Proposition \ref{Prop:E[Z_1^gamma]=infty f.a. gamma>1}.}

For the remainder of this section, we denote by $[\cdot]_u$, $u \in \V$ the canonical shift-operators,
i.e., if $\Psi$ is a function of $(\calZ(v))_{v \in \V}$,
then $[\Psi]_u$ is the same function of $(\calZ(uv))_{v \in \V}$.
For $n \in \N$, introduce the $n$th martingale difference $D_n \defeq Z_n - Z_{n-1} = \sum_{|u|=n-1} L(u) ([Z_1]_u-1)$.

\begin{proof}[Proof of Theorem \ref{Thm:alpha in (1,2)}]
Let $\epsilon > 0$ be as in \eqref{eq:logarithmic moment condition}
and choose $\phi$ as in Lemma \ref{Lem:TV function} with $\delta \defeq 1+\epsilon/2$.
We extend $\phi$ to a function on $\C$ by letting $\phi(x + \imag y) \defeq \phi(x) + \phi(y)$, $x,y \in \R$.
Set $\ell(x) \defeq \phi(x)x^{-\alpha}$ for $x > 0$
and notice that condition \eqref{eq:logarithmic moment condition} implies $C_{\phi\ell} \defeq \E[\phi(|Z_1-1| \vee 1) \ell(|Z_1-1| \vee 1)] < \infty$.

For $t > 0$, we write $D_n^{(t)}$ for the truncated martingale differences
\begin{equation*}
D_n^{(t)} = \sum_{|u|=k-1} L(u) \1_{\{|L(u|_j) \leq t \text{ for } j=0,\ldots,k-1\}} ([Z_1]_u-1)
\end{equation*}
and set $Z_0^{(t)} = 0$ and $Z_n^{(t)} \defeq D_1^{(t)}+\dots+D_n^{(t)}$, $n \in \N$.
It is easy to check that $(Z_n^{(t)})_{n \geq 0}$ is a martingale with respect to $(\F_n)_{n \geq 0}$.
Clearly, $Z_n^{(t)} = Z_n$ for all $n \geq 0$ on the set $\{\sup\!_{u \in \Gen} |L(u)| \leq t\}$.
As in the proof of Proposition 2.1 in \cite{Buraczewski+Kolesko:2014}, we infer from \eqref{eq:many-to-one} that
\begin{align}
\Prob\big(\sup\!_{u \in \Gen} |L(u)| > t\big) &\leq \E[\#\{u: |L(u)| > t \text{ and } |L(u|_k)| \leq t \text{ for all } k < |u|\}]	\notag	\\
&\textstyle
= \E \big[\sum_{n \geq 0} e^{\alpha S_n} \1_{\{S_n < -\log t \text{ and } S_k \geq -\log t \text{ for } k=0,\ldots,n-1 \}}\big]
< t^{-\alpha}.		\label{eq:bound on tails of supremum}
\end{align}
In particular, $\lim_{t \to \infty} \Prob(\sup_{u \in \Gen} |L(u)| > t) = 0$.
Therefore, if we show that $(Z_n^{(t)})_{n \geq 0}$ converges almost surely for every $t > 0$,
then we infer that $(Z_n)_{n \geq 0}$ converges almost surely to some finite limit $Z$.

To prove convergence of $(Z_n^{(t)})_{n \geq 0}$,
we apply the Topchi{\u\i}-Vatutin inequality for martingales \cite[Theorem 1]{Alsmeyer+Roesler:2003} twice
(for the second application
note that $D_k^{(t)}$ conditional on $\F_{k-1}$ is a weighted sum of
independent, centered and $\phi$-integrable random variables)
\begin{align*}
\E[\phi(Z_n^{(t)}-1)]
&\leq 2 \sum_{k=1}^n \E[\phi(D_k^{(t)})]		\\
&\leq 4 \sum_{k=1}^n \E \bigg[\sum_{|u|=k-1} \phi(L(u)([Z_1]_u-1)) \1_{\{|L(u|_j)| \leq t \text{ for } j=1,\ldots,k-1\}}\bigg]		\\
&\leq 8 \sum_{k=1}^n \E \bigg[\sum_{|u|=k-1} \phi(|L(u)([Z_1]_u-1)|) \1_{\{|L(u|_j)| \leq t \text{ for } j=1,\ldots,k-1\}}\bigg],
\end{align*}
where we have used that $\phi(z) \leq 2 \phi(|z|)$ for all $z \in \C$.
Using that
\begin{equation}	\label{eq:ell's inequality}
\ell(|zw|) \leq \ell(|z|) \ell(|w|)^2
\end{equation}
for all $z,w \in \C$ with $|w| \geq 1$ and the change of measure \eqref{eq:many-to-one}, we get
\begin{align}
\E[\phi(Z_n^{(t)}-1)]
&\leq 8 C_{\phi\ell} \sum_{k=0}^{n-1} \E \big[\phi(e^{-S_k}) e^{\alpha S_k} \1_{\{S_i \geq -\log t \text{ for } i=1,\ldots,k\}}\big]	\notag	\\
&\leq 8 C_{\phi\ell} \sum_{k=0}^{\infty} \E \big[\ell(e^{-S_k}) \1_{\{S_i \geq -\log t \text{ for } i=1,\ldots,k\}}\big].	\label{eq:outcome of double TV}
\end{align}
To see that the latter series is finite,
let $\tau_0 \defeq 0$ and let $\tau_n$ denote the $n$th strictly descending ladder epoch for the walk $(S_k)_{k \geq 0}$, $n \in \N$.
Notice that $\E[S_1] \geq 0$ by \eqref{eq:E[S_1]}, hence $\tau_n$ may be infinite with positive probability.
Then, for any $k \geq 0$, there exist unique (random) numbers $n \in \N$ and $j \in \N_0$ such that $\tau_{n-1} \leq k = \tau_{n-1}+j < \tau_n$.
In this case, $S_i \geq - \log t$ for all $i=0,\ldots,k$ if and only if $S_{\tau_{n-1}} \geq -\log t$,
and we infer from \eqref{eq:ell's inequality}
\begin{align*}
&\ell\big(e^{-S_k}\big) \1_{\{S_i \geq -\log t \text{ for } i=1,\ldots,k\}}	\\
&\quad	= \ell\big(e^{-(S_{\tau_{n-1}+j}-S_{\tau_{n-1}})} e^{-S{\tau_{n-1}}}\big) \1_{\{S_{\tau_{n-1}} + \log t \geq 0\}}	\\
&\quad	\leq \ell\big(e^{-(S_{\tau_{n-1}+j}-S_{\tau_{n-1}})} \big) \ell\big(e^{-S{\tau_{n-1}}}\big)^2 \1_{\{S_{\tau_{n-1}} + \log t \geq 0\}}	\\
&\quad	\leq \ell(t)^4 \ell\big(e^{-(S_{\tau_{n-1}+j}-S_{\tau_{n-1}})} \big) \ell\big(e^{-(S{\tau_{n-1}}+\log t)}\big)^2 \1_{\{S_{\tau_{n-1}} + \log t \geq 0\}}.
\end{align*}
We thus infer for the infinite series in \eqref{eq:outcome of double TV}:
\begin{align*}
&\sum_{k=0}^{\infty} \E \big[\ell(e^{-S_k}) \1_{\{S_i \geq -\log t \text{ for } i=1,\ldots,k\}}\big]	\\
&~\leq \ell(t)^4 \,
\E \bigg[\sum_{n=1}^\infty \1_{\{\tau_{n-1}<\infty\}} \ell\big(e^{-(S_{\tau_{n-1}}+\log t)}\big)^2  \1_{\{S_{\tau_{n-1}} + \log t \geq 0\}}
\!\!\!\!\! \sum_{j=0}^{\tau_{n}-\tau_{n-1}-1} \ell\big(e^{-(S_{\tau_{n-1}+j} - S_{\tau_{n-1}})}\big) \bigg]	\\
&~=
\ell(t)^4 \,
\E \bigg[\sum_{n=1}^\infty \1_{\{\tau_{n-1}<\infty\}}\ell\big(e^{-(S_{\tau_{n-1}}+\log t)}\big)^2  \1_{\{S_{\tau_{n-1}} + \log t \geq 0\}}\bigg]
\E\bigg[\sum_{j=0}^{\tau_{1}-1} \ell\big(e^{-S_j}\big) \bigg],
\end{align*}
where we have used the strong Markov property for the random walk $(S_k)_{k \geq 0}$.
Let $\sigma_0 \defeq 0$ and $\sigma_n$ the $n$th weakly ascending ladder epoch of the walk
$(S_k)_{k \geq 0}$, i.e., $\sigma_n \defeq \inf\{k > \sigma_{n-1}: S_k \geq S_{\sigma_{n-1}}\}$, $n \in \N$.
Then the duality lemma gives
\begin{equation}	\label{eq:duality lemma}
\E\bigg[\sum_{j=0}^{\tau_{1}-1} \ell\big(e^{-S_j}\big) \bigg]
= \E\bigg[\sum_{n=0}^{\infty} \ell\big(e^{-S_{\sigma_n}}\big) \bigg].
\end{equation}
By the choice of $\phi$ (see Lemma \ref{Lem:TV function}),
$\ell(e^{-x})$ is decreasing and
$\ell(e^{-x}) \sim c^{-1} x^{-1-\epsilon/2}$ as $x \to \infty$.
Thus, $x \mapsto \ell(e^{-x})\1_{[0,\infty)}$ is directly Riemann integrable.
Now $\E[S_1] \geq 0$ implies $\Prob(\sigma_n<\infty)$ for all $n \in \N$.
Hence $(S_{\sigma_n})_{n \geq 0}$ is a random walk drifting to $+\infty$.
Taken together, we infer that the expectation in \eqref{eq:duality lemma} is finite.
Again from the direct Riemann integrability of $x \mapsto \ell(e^{-x})\1_{[0,\infty)}$,
we conclude that
\begin{equation*}
\sup_{t > 0}\E \bigg[\sum_{n=1}^\infty \1_{\{\tau_{n-1}<\infty\}} \ell\big(e^{-(S_{\tau_{n-1}}+\log t)}\big)^2  \1_{\{S_{\tau_{n-1}} + \log t \geq 0\}}\bigg] < \infty.
\end{equation*}
So far we have shown that there is a constant $C > 0$, not depending on $t$,
such that
\begin{equation}	\label{eq:bounds on phi-moments}	\textstyle
\sup_{n \geq 1} \E[\phi(Z_n^{(t)}-1)] \leq C \ell(t)^4
\end{equation}
for all $t > 0$. This implies that $Z_n^{(t)} \to Z^{(t)}$ almost surely for some random variable $Z^{(t)}$
and, upon letting $t \to \infty$, also $Z_n \to Z$ almost surely for $Z \defeq \lim_{t \to \infty} Z^{(t)}$.
What is more,
\begin{align*}
\Prob(|Z_n-1| > t) &	\textstyle
\leq	\Prob(\phi(|Z_n-1|) > \phi(t),\;\sup_{u \in \Gen} |L(u)| \leq t) + \Prob(\sup_{u \in \Gen} |L(u)| > t)	\\
&	\textstyle
\leq	\Prob(\phi(|Z_n^{(t)}-1|) > \phi(t)) + t^{-\alpha}
\leq	\phi(t)^{-1} \sup_{n \geq 1} \E[\phi(Z_n^{(t)}-1)] + t^{-\alpha}	\\
&	\textstyle
\leq	t^{-\alpha} (C \ell(t)^3 +1)
\end{align*}
for all sufficiently large $t$. As $\ell(t)$ is of the order $\log^{1+\epsilon/2} t$ as $t \to \infty$,
the bound above implies that $(|Z_n-1|^p)_{n \geq 0}$ is uniformly integrable for all $p<\alpha$.
Consequently, $Z_n \to Z$ in $L^p$ for all $p<\alpha$. In particular, $\E[Z] = 1$.
\end{proof}

Proposition \ref{Prop:alpha=2} and Proposition \ref{Prop:E[Z_1^gamma]=infty f.a. gamma>1}
can be proved using minor modifications of the corresponding results in \cite{Iksanov+Meiners:2015}.
For the reader's convenience, we sketch the corresponding arguments in the given context.

Both propositions are based on the following lemma.

\begin{lemma}	\label{Lem:tail bounds}
Suppose that $\Prob(N < \infty)=1$ and that \eqref{eq:characteristic index} holds.
Further, assume that
$\E\big[\sum_{|u|=1} |L(u)|^\alpha \log(|L(u)|)\big] \in (-\infty,0)$
and $\E[W_1 \log_+ W_1] < \infty$,
or that \eqref{eq:m(vartheta)<infty} holds.
Then $Z_n \to Z$ in probability as $n \to \infty$ implies $\Prob(|Z| \geq t) = o(t^{-p})$ as $t \to \infty$
and, in particular, $\E[|Z_1|^p] < \infty$ for every $p \in (1,\alpha)$.
\end{lemma}

The proof of the lemma is lengthy and follows along the lines of
the proofs of \cite[Lemma 4.9]{Alsmeyer+Meiners:2013} and \cite[Lemma 4.7]{Iksanov+Meiners:2015}.
We will therefore only give a sketch of the proof.

\begin{proof}[Sketch of the proof]
First notice that if $Z_n \to Z$ in probability as $n \to \infty$,
then $Z$ satisfies
\begin{equation}	\label{eq:endogenous FPE}	\textstyle
Z = \sum_{|u|=n} L(u) [Z]_u	\quad	\text{almost surely}
\end{equation}
for every $n \in \N$.
This means that $Z$ is a fixed point of a smoothing transformation.
The proof of Lemma \ref{Lem:tail bounds} is based on a comparison
of the survival probability $\Prob(|Z| > t)$ with the
Laplace transform $\varphi$ at $0$ solving the functional equation of a suitable smoothing transform.
To be more precise, there exists a probability measure on $[0,\infty)$, non-degenerate at $0$,
such that its Laplace transform $\varphi$ satisfies
\begin{equation}	\label{eq:functional equation}	\textstyle
\varphi(t) = \E\big[\prod_{|u|=1} \varphi(|L(u)|^\alpha t)\big],	\quad	t \geq 0.
\end{equation}
Indeed, $\varphi$ is the Laplace transform of a fixed point of a smoothing transformation on the nonnegative halfline
with tilted weights $|L(u)|^{\alpha}$, $|u|=1$.
Further, $\varphi$ is such that $1-\varphi(t)$ is regularly varying of index $1$ at $0$.
These facts are summarized in \cite{Alsmeyer+al:2012}, see in particular Proposition 2.1 and Theorem 3.1 there.
As in \cite[Section 3.5]{Iksanov+Meiners:2015},
using multiplicative martingales and the theory of independent, infinitesimal triangular arrays,
one can deduce that $\Prob(|Z| > t) = o(1-\varphi(t^{-\alpha}))$ as $t \to \infty$.
Thus, $\Prob(|Z| > t) = o(t^{-p})$ as $t \to \infty$ for every $p \in (1,\alpha)$.
In particular, for any $p \in (1,\alpha)$, we have $\E[|Z|^p]<\infty$
and thus, by standard martingale theory, $\E[|Z_1|^p] = \E[|\E[Z|\F_1]|^p] \leq \E[\E[|Z|^p|\F_1]] = \E[|Z|^p] < \infty$.
\end{proof}

Proposition \ref{Prop:alpha=2} can be proved as Theorem 2.3 in \cite{Iksanov+Meiners:2015}.
We therefore keep the presentation short here.
\begin{proof}[Proof of Proposition \ref{Prop:alpha=2}]
Suppose that $\Prob(N < \infty)=1$ and that \eqref{eq:characteristic index} holds with $\alpha = 2$
and that one of the additional conditions holds.
Further, assume for a contradiction that $Z_n \to Z$ in probability as $n \to \infty$.
Then we can apply Lemma \ref{Lem:tail bounds}
and deduce that $\E[|Z|^p]<\infty$ for every $p \in (1,\alpha)$.
Standard martingale theory gives $\E[|Z_n-Z|^p] \to 0$ as $n \to \infty$ for each such $p$.
On the other hand, from the Burkholder-Davis-Gundy inequality \cite[Theorem 11.3.1]{Chow+Teicher:1997} and Jensen's inequality
for the concave function $x \mapsto x^{p/2}$ for $x \geq 0$,
we get as in the proof of Theorem 2.3 in \cite{Iksanov+Meiners:2015} that there exists a constant $c_p > 0$ such that
\begin{align*}
\E[|\Real(Z_n)-1|^p + |\Imag(Z_n)|^p]	
&	\textstyle
~\geq c_p \E\big[(\sum_{k=1}^n |\Real(D_k)|^2)^{p/2}+(\sum_{k=1}^n |\Imag(D_k)|^2)^{p/2}\big]	\\
&	\textstyle
~\geq c_p n^{p/2-1} \E\big[(\sum_{k=1}^n |\Real(D_k)|^p)+(\sum_{k=1}^n |\Imag(D_k)|^p)\big].
\end{align*}
Here, using that given $\F_{k-1}$, $D_k$ is a weighted sum of centered i.i.d.\ random variables,
we can again apply the Burkholder-Davis-Gundy inequality and then Jensen's inequality on $\{W_{k-1} > 0\}$ to infer
\begin{align*}	\textstyle
&\E[|\Real(D_k)|^p+|\Imag(D_k)|^p]	\\
&	\textstyle
~\geq c_p \E\big[(\sum_{|u|=k-1} \Real(L(u)([Z_1]_u-1))^2)^{p/2}+(\sum_{|u|=k-1} \Imag(L(u)([Z_1]_u-1))^2)^{p/2}\big]	\\
&	\textstyle
~\geq c_p 2^{p/2-1} \E\big[(\sum_{|u|=k-1} (\Real(L(u)([Z_1]_u-1))^2+\Imag(L(u)([Z_1]_u-1))^2))^{p/2}\big]	\\
&	\textstyle
~= c_p 2^{p/2-1} \E\big[(\sum_{|u|=k-1} |L(u)([Z_1]_u-1)|^2)^{p/2}\big]	\\
&	\textstyle
~\geq c_p 2^{p/2-1} \E[|Z_1-1|^p] \, \E[W_{k-1}^{p/2}].
\end{align*}
Consequently,
\begin{align}
\E[|Z_n-1|^p]
&	\textstyle
\geq 2^{p/2-1} \E[|\Real(Z_n)-1|^p + |\Imag(Z_n)|^p]	\notag	\\
&	\textstyle
\geq c_p^2 2^{p-2} \E[|Z_1-1|^p] n^{p/2-1} \sum_{k=0}^{n-1} \E[W_{k}^{p/2}].	\label{eq:lower bound E[|Z_n-1|^p]}
\end{align}
Condition (i) implies that $W_n \to W$ in $L^1$, see \cite{Lyons:1997}.
Hence the lower bound in \eqref{eq:lower bound E[|Z_n-1|^p]} is of the order $n^{p/2}$ which tends to $+\infty$ as $n \to \infty$.
Condition (ii) implies that $n^{p/4} W_n^{p/2}$, $n \in \N$ converges in distribution as $n \to \infty$
to a non-degenerate limit and is also uniformly integrable, see \cite[Theorem 1.1]{Aidekon+Shi:2014} and \cite[Remark 4.8]{Iksanov+Meiners:2015}.
Thus the lower bound in \eqref{eq:lower bound E[|Z_n-1|^p]} is of the order $n^{p/4}$ and again diverges as $n \to \infty$.
\end{proof}

\begin{proof}[Proof of Proposition \ref{Prop:E[Z_1^gamma]=infty f.a. gamma>1}]
The proposition follows from Lemma \ref{Lem:tail bounds} via contraposition.
\end{proof}

\section{Results for higher dimensions}	\label{sec:similarities}

As already pointed out in the introduction,
to a large extent, our interest in the problem of complex martingale convergence in the branching random walk comes from
its significance in the fixed-point theory for smoothing transformations.
As this theory has applications to problems that go (with regard to the dimension) beyond the complex case,
we will explain how the results obtained above can be extended.

To be precise, fix a dimension $d \in \N$ and let $\Sim$ denote the set of real $d \times d$ similarity matrices.
A similarity matrix is a matrix that can be written as the product of a positive scaling factor and an orthogonal $d \times d$ matrix.
Now suppose that $\calZ = \sum_{u=1}^N \delta_{L(u)}$ is a point process on $\Sim$, i.e., the $L(u)$, $u=1,\ldots,N$ are similarity matrices.
A fixed point of the smoothing transform associated with $\calZ$ is a $d$-dimensional random vector $X$
satisfying
\begin{equation}	\label{eq:fixed-point equation of smoothing trafo}	\textstyle
X \stackrel{\mathrm{law}}{=} \sum_{u=1}^N L(u) X_u
\end{equation}
where $X_1,X_2,\ldots$ are i.i.d.\ copies of $X$ and independent of $\calZ$.
An important problem arising when solving \eqref{eq:fixed-point equation of smoothing trafo}
is the following. Take independent copies $\calZ(u)$, $u \in \V$ of $\calZ$ on a suitable probability space $(\Omega,\A,\Prob)$
and define $\Gen_n$, $\Gen$, $\F_n$ in obvious analogy to the corresponding objects defined in Section \ref{sec:main results}.
Define $L(\varnothing)$ to be the $d \times d$ identity matrix,
and, for $uj \in \Gen$, define recursively
$L(uj) \defeq L(u) [L(j)]_{u}$.
Now suppose that the matrix $\E[\sum_{|u|=1} L(u)]$ has finite entries only and that
it has a right eigenvector $0 \not = w \in \R^d$ to the eigenvalue $1$.
Then the sequence $(Z_n w)_{n \in \N_0}$ defined via
\begin{equation*}	\textstyle
Z_n w \defeq \sum_{|u|=n} L(u) w,	\quad	n \in \N_0
\end{equation*}
defines a $d$-dimensional martingale with respect to $(\F_n)_{n \geq 0}$.
In slight abuse of common notation, we write $|\cdot|$ not only for the standard Euclidean norm in $\R^d$ but also for the usual matrix norm.
Since we only work with similarity matrices, this should cause no confusion.
Then condition \eqref{eq:characteristic index} makes perfect sense in the given situation,
and the following result can be proved along the lines of the proof of Theorem \ref{Thm:alpha in (1,2)}:

\begin{theorem}	\label{Thm:alpha in (1,2), similarity case}
Suppose that \eqref{eq:characteristic index} holds with $\alpha \in (1,2)$
and that \eqref{eq:logarithmic moment condition} holds with $Z_1$ replaced by $Z_1 w$.
Then $(Z_n w)_{n \geq 0}$ converges almost surely and in $L^p$ for every $p < \alpha$
to a non-degenerate limit $Z^w$.
\end{theorem}

This improves Proposition 1.1(c) in \cite{Meiners+Mentemeier:2017}
in two ways. First of all, the assumptions on finite absolute moments of $Z_1 w$ are relaxed.
Second, the theorem above includes the boundary case $m'(\alpha)=0$,
which is not covered in \cite{Meiners+Mentemeier:2017}.	

Also, with $W_n \defeq \sum_{|u|=n} |L(u)|^{\alpha}$, $n \in \N_0$,
the analog of Lemma \ref{Lem:tail bounds} holds in the given context
and thus allows to conclude the analogs of
Propositions \ref{Prop:alpha=2} and \ref{Prop:E[Z_1^gamma]=infty f.a. gamma>1}
with $Z_n$ replaced by $Z_n^w$, $n \in \N_0$.
We refrain from reformulating the corresponding results in the more general context.

\begin{appendix}
\section{Auxiliary results}	\label{sec:appendix}

\begin{lemma}	\label{Lem:TV function}
Let $\alpha \in (1,2)$, $\delta > 0$.
Then there is an even convex function $\phi: \R \to [0,\infty)$
with $\phi(0)=0$
having a concave derivative on $(0,\infty)$
such that $\ell(x) \defeq \phi(x)x^{-\alpha}$, $x > 0$ 
is increasing and satisfies the following assertions:
\begin{itemize}	\itemsep-1pt
	\item[(i)]	
		For all $x > 0$, we have $\ell(x^{-1}) = \ell(x)^{-1} > 0$.
	\item[(ii)]
		There exists a constant $c > 0$ such that $\ell(x) \sim c \log^{\delta}(x)$ as $x \to \infty$.
	\item[(iii)]
		 For all $x\geq1$ and $y > 0$, we have $\ell(xy) \leq \ell(x)^2\ell(y)$.
\end{itemize}
\end{lemma}
\begin{proof}
We set
$\varepsilon(u) \defeq \delta u_0^{-1} \1_{[0,u_0]}(|u|) + \delta |u|^{-1} \1_{(u_0,\infty)}(|u|)$
for some $u_0 > 0$ to be specified below,
and
\begin{equation*}	\textstyle
\ell(x)	\defeq \exp \big(\int_0^{\log x} \varepsilon(u) \, \du \big),	\quad	x > 0
\end{equation*}
where the integral has to be understood as an (oriented) Riemann integral.
We then define $\phi(0) \defeq 0$ and $\phi(x) \defeq |x|^{\alpha} \ell(|x|)$ for $x \not = 0$.
Then $\ell$ satisfies (i) since $\varepsilon$ is symmetric around $0$.
From $\varepsilon(u) = \delta u^{-1}$ for all $u \geq u_0$ we conclude that (ii) holds.
For the proof of (iii), first notice that
since $\varepsilon$ is decreasing on $[0,\infty)$, the integral $\int_0^x \varepsilon(u) \, \du$
is subadditive as a function of $x \geq 0$. Consequently, $\ell(xy) \leq \ell(x) \ell(y) \leq \ell(x)^2\ell(y)$ for all $x,y \geq 1$.
Now suppose $x \geq 1$ and $y < 1$. We distinguish two cases.
If $xy < 1$, then
\begin{equation*}	\textstyle
\int_0^{\log(xy)} \varepsilon(u) \, \du = \int_{\log y}^{\log x+ \log y} \varepsilon(u) \, \du - \int_{\log y}^0 \varepsilon(u) \, \du
\leq \int_{0}^{\log x} \varepsilon(u) \, \du + \int_0^{\log y} \varepsilon(u) \, \du,
\end{equation*}
where we have used that $\varepsilon$ is symmetric and decreasing on $[0,\infty)$.
Again, we conclude that $\ell(xy) \leq \ell(x) \ell(y) \leq \ell(x)^2\ell(y)$.
Next, suppose $xy \geq 1$. Then
\begin{equation*}	\textstyle
\int_0^{\log(xy)} \varepsilon(u) \, \du \leq \int_0^{\log x} \varepsilon(u) \, \du \leq 2 \int_0^{\log x} \varepsilon(u) \, \du + \int_0^{\log y} \varepsilon(u) \, \du,
\end{equation*}
hence $\ell(xy) \leq \ell(x)^2\ell(y)$.

Finally, we have to show that we can choose $u_0 > 0$ such that $\phi$ is convex on $\R$ with concave derivative on $(0,\infty)$.
Clearly, $\phi$ is continuously differentiable with derivative
\begin{align*}
\phi'(x)	&= x^{\alpha-1} \ell(x) (\alpha + \varepsilon(\log x)),	\quad	t > 0.
\end{align*}
As $\varepsilon$ is smooth on $(0,\infty) \setminus \{u_0\}$, so is $\phi$,
and we get for the higher order derivatives:
\begin{align*}
\phi''(x)	&= x^{\alpha-2} \ell(x) (\alpha (\alpha-1) + (2\alpha-1) \varepsilon(\log x) + \varepsilon^2(\log x) + \varepsilon'(\log x)),	\\
\phi'''(x)	&= x^{\alpha-3} \ell(x) \big(\alpha(\alpha-1)(\alpha-2) + p_0(\varepsilon(\log x)) + p_1(\varepsilon'(\log(x))) + p_2(\varepsilon''(\log(x)))\big)
\end{align*}
for $x > 0$, $x \not = u_0$,
where $p_0,p_1,p_2$ are polynomials with $p_j(0)=0$ for $j=0,1,2$ and coefficients depending only on $\alpha$.
Consequently, there exists a constant $\eta > 0$
such that $\phi''(x) > 0$ and $\phi'''(x) < 0$ for all $x > 0$, $x \not = u_0$ such that
$\max\{|\varepsilon(x)|,|\varepsilon'(x)|,|\varepsilon''(x)|\} \leq \eta$.
Now fix $u_0 > 0$ so large that $\max\{|\varepsilon(u)|,|\varepsilon'(u)|,|\varepsilon''(u)|\} \leq \eta$ for all $u \geq u_0$.
Then $\phi'''(x) < 0$ for all $x > 0$, $x \not = e^{u_0}$, hence $\phi''$ is strictly decreasing on $(0,e^{u_0})$ and $(e^{u_0},\infty)$.
From the explicit expression for $\phi''$ above, we conclude that $\phi''(u_0-) > \phi''(u_0^+)$
(the difference between these expressions is given exactly by the difference of the limits of $\varepsilon'(\log x)$ as $x \uparrow e^{u_0}$, which is $0$,
and as $x \downarrow e^{u_0}$, which is $-\delta u_0^{-2} < 0$).
Thus $\phi'$ is (strictly) concave.
Analogously, we infer that $\phi$ is (strictly) convex on $[0,\infty)$.
\end{proof}

\end{appendix}


\ACKNO{We are grateful to Zakhar Kabluchko for giving us valuable input.}

\bibliographystyle{plain}
\bibliography{BRW}

\begin{thebibliography}{10}

\bibitem{Aidekon+Shi:2014}
Elie A{\"{\i}}d{\'e}kon and Zhan Shi.
\newblock The {S}eneta-{H}eyde scaling for the branching random walk.
\newblock {\em Ann. Probab.}, 42(3):959--993, 2014.

\bibitem{Alsmeyer+al:2012}
Gerold Alsmeyer, J.~D. Biggins, and Matthias Meiners.
\newblock The functional equation of the smoothing transform.
\newblock {\em Ann. Probab.}, 40(5):2069--2105, 2012.

\bibitem{Alsmeyer+Iksanov:2009}
Gerold Alsmeyer and Alexander Iksanov.
\newblock A log-type moment result for perpetuities and its application to
  martingales in supercritical branching random walks.
\newblock {\em Electron. J. Probab.}, 14:no. 10, 289--312, 2009.

\bibitem{Alsmeyer+Meiners:2013}
Gerold Alsmeyer and Matthias Meiners.
\newblock Fixed points of the smoothing transform: two-sided solutions.
\newblock {\em Probab. Theory Related Fields}, 155(1-2):165--199, 2013.

\bibitem{Alsmeyer+Roesler:2003}
Gerold Alsmeyer and Uwe R{\"o}sler.
\newblock The best constant in the {T}opchii-{V}atutin inequality for
  martingales.
\newblock {\em Statist. Probab. Lett.}, 65(3):199--206, 2003.

\bibitem{Barral+al:2010}
Julien Barral, Xiong Jin, and Beno{\^{\i}}t Mandelbrot.
\newblock Convergence of complex multiplicative cascades.
\newblock {\em Ann. Appl. Probab.}, 20(4):1219--1252, 2010.

\bibitem{Biggins:1977}
J.~D. Biggins.
\newblock Martingale convergence in the branching random walk.
\newblock {\em J. Appl. Probability}, 14(1):25--37, 1977.

\bibitem{Biggins:1992}
J.~D. Biggins.
\newblock Uniform convergence of martingales in the branching random walk.
\newblock {\em Ann. Probab.}, 20(1):137--151, 1992.

\bibitem{Buraczewski+Kolesko:2014}
Dariusz Buraczewski and Konrad Kolesko.
\newblock Linear stochastic equations in the critical case.
\newblock {\em J. Difference Equ. Appl.}, 20(2):188--209, 2014.

\bibitem{Chow+Teicher:1997}
Yuan~Shih Chow and Henry Teicher.
\newblock {\em Probability theory}.
\newblock Springer Texts in Statistics. Springer-Verlag, New York, third
  edition, 1997.
\newblock Independence, interchangeability, martingales.

\bibitem{Iksanov+Meiners:2015}
Alexander Iksanov and Matthias Meiners.
\newblock Fixed points of multivariate smoothing transforms with scalar
  weights.
\newblock {\em ALEA Lat. Am. J. Probab. Math. Stat.}, 12(1):69--114, 2015.

\bibitem{Lacoin+al:2015}
Hubert Lacoin, R{\'e}mi Rhodes, and Vincent Vargas.
\newblock Complex {G}aussian multiplicative chaos.
\newblock {\em Comm. Math. Phys.}, 337(2):569--632, 2015.

\bibitem{Lyons:1997}
Russell Lyons.
\newblock A simple path to {B}iggins' martingale convergence for branching
  random walk.
\newblock In {\em Classical and modern branching processes ({M}inneapolis,
  {MN}, 1994)}, volume~84 of {\em IMA Vol. Math. Appl.}, pages 217--221.
  Springer, New York, 1997.

\bibitem{Madaule+al:2015}
Thomas Madaule, Rémi Rhodes, and Vincent Vargas.
\newblock Continuity estimates for the complex cascade model on the phase
  boundary, 2015.
\newblock arXiv:1502.05655.

\bibitem{Mandelbrot:1972}
Benoit Mandelbrot.
\newblock Possible refinement of the lognormal hypothesis concerning the
  distribution of energy dissipation in intermittent turbulence.
\newblock In M.~Rosenblatt and C.~Van~Atta, editors, {\em Statistical models
  and turbulence. Proceedings of a {S}ymposium held at the {U}niversity of
  {C}alifornia, {S}an {D}iego ({L}a {J}olla), {C}alif., {J}uly 15--21, 1971},
  pages 333--351. Springer-Verlag, Berlin-New York, 1972.
\newblock Lecture Notes in Physics, Vol. 12.

\bibitem{Mandelbrot:1974}
Benoit Mandelbrot.
\newblock Multiplications al\'eatoires it\'er\'ees et distributions invariantes
  par moyenne pond\'er\'ee al\'eatoire.
\newblock {\em C. R. Acad. Sci. Paris S\'er. A}, 278:289--292, 1974.

\bibitem{Meiners+Mentemeier:2017}
Matthias Meiners and Sebastian Mentemeier.
\newblock Solutions to complex smoothing equations.
\newblock {\em \emph{To appear in} Probab. Theory Related Fields}, 2017+.
\newblock arXiv:1507.08043.

\end{thebibliography}

\end{document}